\documentclass[12pt]{amsart}
\usepackage{amssymb}
\usepackage{amsmath}
\usepackage{graphicx}
\usepackage{enumerate}
\usepackage{latexsym}
\usepackage{amsxtra}
\usepackage{amsfonts}
\usepackage{cmmib57}
\newcommand{\age}{{\mathcal A}}
\newcommand{\profile}{\varphi}

\newcommand{\N}{\mathbb N}
\newcommand{\C}{\mathbb C}
\newcommand{\Z}{\mathbb Z}
\newcommand{\Q}{\mathbb Q}

\newcommand{\T}{\mathbb T}
\newcommand{\K}{\mathbb K}
\newcommand{\hilbert}{{\mathcal H}}
\newtheorem{definition}{{\bf Definition}}[section]

\newtheorem{theorem}[definition]{{\bf Theorem}}
\newtheorem{conjectures}[definition]{{\bf Conjectures}}

\newtheorem{corollary}[definition]{{\bf Corollary}}
\newtheorem{proposition}[definition]{\noindent {\bf Proposition}}
\newtheorem{lemma}[definition]{\noindent {\bf Lemma}}

\newtheorem{claim}[definition]{\noindent {\bf Claim}}

\newtheorem{example}[definition]{\noindent {\bf Example}}

\newtheorem{fact}[definition]{\noindent{\bf Fact}}

\newtheorem{problem}{Problem}

\newtheorem{remark}[definition]{\noindent {\bf Remark}}

\def\endproof{\hfill {\kern 6pt\penalty 500
   \raise -0pt\hbox{\vrule \vbox to5pt {\hrule width 5pt
  \vfill\hrule}\vrule}}}

\def\References[definition]{\bf R\'ef\'erences}
\def\endproof{\hfill {\kern 6pt\penalty 500
\raise -0pt\hbox{\vrule \vbox to5pt {\hrule width 5pt
\vfill\hrule}\vrule}}}
\def\GenShuffle#1#2#3{\mathbin{
      \hbox{\vbox{
        \hbox{\vrule
              \hskip#2
              \vrule height#1 width 0pt
               }%
        \hrule\vskip#3}%
             \vbox{
        \hbox{\vrule
              \hskip#2
              \vrule height#1 width 0pt
               \vrule }%
        \hrule\vskip#3}%
}}}
\def\shuffle{\,{\mathchoice{\GenShuffle{5pt}{3.5pt}{1pt}}%
                           {\GenShuffle{4pt}{3pt}{1pt}}%
                           {\GenShuffle{3pt}{2pt}{0.7pt}}%
                           {\GenShuffle{2pt}{1.5pt}{0.5pt}}}\,}%

\begin{document}

\title[A conjecture of P.J. Cameron]{When the  orbit algebra of group is an integral domain? Proof of a conjecture of P.J. Cameron}
\thanks{Research done under the auspices of Intas programme 03-51-4110 "Universal algebra and lattice theory" }
\author{Maurice Pouzet}
\address{ICJ, Math\'ematiques, Universit\'e Claude-Bernard Lyon1,
Domaine de Gerland, B\^at. Recherche [B], 50 avenue Tony-Garnier, F69365  Lyon cedex 07, France, e-mail: pouzet@univ-lyon1.fr}



\keywords {Relational structures, ages, counting functions, oligomorphic groups, age algebra, Ramsey theorem, integral domain. }
\subjclass{03 C13, 03 C52,  05 A16, 05 C30, 20 B27}

\date{\today}

\begin{abstract} P.J.Cameron introduced the orbit algebra of a permutation group and conjectured that this algebra is an integral domain if and
only if the group has no finite orbit. We prove that this conjecture holds and in fact that the age algebra of a
relational structure $R$ is an integral domain if and only if $R$ is age-inexhaustible. We deduce these results from a
combinatorial lemma asserting that if a product of two  non-zero elements of a set algebra   is zero then there is a
finite  common tranversal of their supports. The proof is built on Ramsey theorem and the integrity of a shuffle algebra.
 \end{abstract}\maketitle

\section* { Introduction}
In 1981,  P.J.Cameron \cite{cameron81} (see also \cite {cameron3} p.86) 
associated a graded algebra $A[G]$ to a permutation group  $G$ acting on  an infinite set $E$. He formulated two conjectures on the integrity of this algebra. The purpose of this paper is to present a solution to the first  of these conjectures. Consequences on the enumeration of  finite substructures of a given structure are mentionned. Some problems are stated.
\subsection{The conjectures}
Here is the content of these conjectures, freely adapted from  Cameron's web page (see Problem 2 \cite{cameron4}). 
The  graded algebra $A[G]$ is the direct sum $$\sum_{n<\omega} A[G]_{n}$$ where $A[G]_{n}$ is the set of
all $G$-invariant functions $f$ from the set $[E]^{n}$ of $n$-element subsets of $E$ into  the field $\C$ of complex numbers. 
Multiplication is defined by the rule that if $f \in A[G]_{m}$, $g \in A[G]_{n}$ and $Q$ is an $(m+n)$-element subset of $E$
then 
\begin{equation}\label{eq:product}(fg)(Q):= \sum_{P\in [Q]^{m}}  f(P)g(Q\setminus P) 
\end{equation}
As shown by Cameron, the constant function $e$ in $A[G]_{1}$ (with value $1$ on every one element set) is not a zero-divisor (see Theorem \ref{cameron} below). The group $G$ is {\it entire} if $A[G]$
is an integral domain, and {\it strongly entire} if $A[G]/eA[G]$ is an integral domain. \\
\begin{conjectures}
$G$ is (strongly) entire if and only if it has no finite orbit on $E$. 
 \end{conjectures}
The condition that $G$ has no finite orbit on $E$ is necessary.  We prove that it suffices for $G$ to be
entire.  As it turns out, our proof extends to
the algebra of an age, also invented by Cameron \cite{cameron4}.
\subsection{The  algebra of an age}
A {\it relational structure} is a realization of a language whose non-logical symbols are predicates. This is a pair $R:=(E,(\rho_{i})_{i\in I})$ made of a set $E$ and a family of $m_i$-ary relations $\rho_i$ on $E$.  The set $E$ is the \emph{domain} or \emph{base} of $R$; the family $\mu:= (m_i)_{i\in I}$ is the \emph{signature} of $R$. The \emph{substructure induced by $R$ on a subset $A$} of $E$, simply called the \emph{restriction of $R$ to $A$}, is the relational structure $R_{\restriction A}:= (A, (A^{m_i}\cap \rho_i)_{i\in I})$.  Notions of \emph{isomorphism}, as well as \emph{isomorphic type}, are defined in natural way (see Subsection  \ref{sec:proof}).

 A map
$f: [E]^m\rightarrow \C$, where $m$ is a non negative integer,  is \emph{$R$-invariant} if $f(P)=f(P')$ whenever the
restrictions $R_{\vert P}$ and $R_{\vert P'}$ are isomorphic. The $R$-invariant maps can be multiplied. Indeed, it is not difficult to show that if $f: [E]^m\rightarrow \C$ and $g:[E]^n\rightarrow \C$ are $R$-invariant,  the product defined by Equation (\ref{eq:product}) is $R$-invariant. Equipped with this multiplication,  the  $\C$-vector space spanned by the  $R$-invariant maps becomes a graded  algebra, the \emph{age algebra of $R$}, that we denote by $\C.\age(R)$. The name, coined by Cameron, comes from the notion of \emph {age} defined  by Fra{\"\i}ss{\'e} \cite{fraissetr}. Indeed, the {\it age} of 
$R$ is the collection
${\mathcal{A}}(R)$ of   substructures of
$R$ induced on the finite subsets of $R$, isomorphic substructures being identified. And it can be  shown that two relational structures with the same age yields the  same algebra (up to an isomorphism of graded algebras). 

The algebra associated to a group is a special case of age algebra. Indeed, to  a permutation group $G$ acting on $E$ we may associate a relational structure $R$ with base $E$ such that the $G$-invariant maps coincide with the $R$-invariant maps. 

Our  criterium  for the integrity of the age algebra is based on the notion of kernel: 
  
The \emph{ kernel } of  a relational structure $R$ is
the subset $K(R)$ of $x\in 
E$ such that  ${\mathcal{A}}(R_{\vert E\setminus \{x\}})\not = {\mathcal{A}}(R)$. 

 The emptyness of the kernel $R$ is a necessary condition for the integrity of the age algebra. Indeed, if $K(R)\not=\emptyset$, pick $x\in K(R)$ and $F\in [E]^{<\omega}$ such that $R_{\restriction F}\in {\mathcal{A}}(R)\setminus {\mathcal{A}}(R_{\vert E\setminus \{x\}})$. Let $P\in[E]^{<\omega}$. Set $f(P):=1$ if $R_{\restriction P}$ is isomorphic to $R_{\restriction F}$, otherwise set $f(P):=0$. Then $f^2:=f f=0$.

\begin{theorem}\label{agealgebranozero} Let $R$ be a relational structure with possibly infinitely many non isomorphic types of $n$-element substructures. The age algebra $\C.\age(R)$ is an integral domain if and only if the kernel of $R$ is empty. 
\end{theorem}

The application to the conjecture of Cameron is immediate. Let $G$ be a permutation group acting on $E$ and let $R$ be a relational structure  encoding  $G$. Then, the kernel of $R$ is the union of  the finite $G$-orbits of the one-element sets. Thus,  if  $G$  has no finite orbit, the kernel of $R$  is empty. Hence from Theorem \ref {agealgebranozero},  $A[G]$ is an integral domain,  as conjectured by Cameron.  

We deduce Theorem \ref {agealgebranozero} from a combinatorial  property of a set algebra over a field (Theorem \ref{main} below). This property does not depends upon the field, provided that its characteristic is zero. The proof we give in Section \ref{sec:proof} is an extension of our $1970$ proof that the profile of an infinite relational structure does not decrease (see Theorem \ref{increaseinfinite} below). The key tool we used then was Ramsey's theorem presented in terms of a property of \emph{almost-chainable relations}. Here, these relations are replaced  by \emph{$F-L$-invariant relational structures}, structures which appeared, under other names,  in  several of our papers (see \cite {pouzettr}, \cite{pouzet79}, \cite{pouzetri}). The final step is reminiscent of the proof of the integrity of a shuffle algebra.

We introduced the notion of kernel in\cite{pouzettr} and studied it in several papers \cite{pouzet79} \cite{pouzetrpe}, \cite{pouzetri} and \cite{pouzetsobranisa}. As it is easy to see (cf \cite{pouzet79}\cite{pouzetsobranisa}), the kernel of a relational structure $R$ is empty if and only if
for every finite subset $F$ of $E$ there is a disjoint subset $F'$ such that the
restrictions $R_{\vert F}$ and $R_{\vert F'}$ are isomorphic. Hence, 
relational structures with empty kernel are those for which their age has the \emph{disjoint embedding property}, meaning that two arbitrary members of the age can be embedded into a third in such a way that their domain are disjoint. In Fra{\"\i}ss{\'e}'s terminology, ages with the disjoint embedding property are said \emph{inexhaustible} and relational structures whose age  is inexhaustible are said \emph{age-inexhaustible}; we say that relational structures with finite kernel are \emph{almost age-inexhaustible}. \footnote{In order to agree with Fra{\"\i}ss{\'e}'s terminology, we disagree with the terminology of our papers, in which inexhaustibility, resp. almost inexhaustibility, is used for relational structures with empty, resp. finite, kernel, rather than for their ages.}

\subsection{A transversality property of the  set  algebra}
Let $\K$ be  a field with characteristic zero. Let $E$ be a set and let $[E]^{<\omega}$ be the set of finite subsets of $E$
(including the empty set $\emptyset$). Let $\K^{[E]^{<\omega}}$ be the set of maps $f:[E]^{<\omega}\rightarrow \K$. Endowed with the usual
addition and scalar multiplication of maps, this set is a vector space over $\K$. Let
$f,g\in
\K^{[E]^{<\omega}}$ and  $Q\in [E]^{<\omega}$.  
Set:  
\begin{equation} \label{eq:setproduct}
fg(Q) = \sum_{P \in [Q]^{<\omega}} f(P)g(Q\setminus P)
\end{equation}

With this operation added, the above set becomes a  $\K$-algebra. This algebra is commutative and it has a
unit, denoted by $1$. This is the map taking the value $1$ on the empty set and the value $0$
everywhere else. The \emph{set algebra} is the subalgebra 
 made of maps $f$ such that $f(P)=0$ for every $P\in [E]^{<\omega}$ with $\vert P\vert $ large enough. This algebra is graded, the homogeneous component of degree $n$ being made of maps which take the value $0$ on every subset of size different from $n$ (see Cameron \cite {cameron1}). If $f$ and  $g$ belong to two homogeneous components, their product is given by  Equation (\ref{eq:product}),   thus an age algebra, or a group algebra, $A$, as previously defined,  is a  subalgebra of this set algebra. The set algebra is far from to be an integral domain. But, with the notion of degree, the integrity of  $ A$ will reduce to the fact that if $m$ and  $n$ are  two non negative integers and  $f:[E]^m\rightarrow \K$, $f:[E]^n\rightarrow \K$ are two non-zero maps belonging to $ A$, their product $fg$ is non zero. 

Let ${\mathcal H}$ be a family of subsets of $E$, a subset $T$ of $E$  is a  {\it
transversal} of $\mathcal H$ if $F\cap T\not = \emptyset$ for every $F\in \mathcal H$; the \emph{transversality} of  ${\mathcal H}$, denoted $\tau ({\mathcal H})$, is the minimum of the cardinalities  (possibly
infinite) of  transversals of $\mathcal H$. We make the convention that $\tau ({\mathcal H})=0$ if $\mathcal H$ is empty. 

Let  $f:[E]^m\rightarrow \K$, denote $supp(f):=\{P\in[V]^m: f(P)\not =0\}$. 

Here is our combinatorial result:
\begin{theorem}\label{main} Let $m,n$ be two non negative integers. There is an integer $t$ such that for every set $E$ with at least $m+n$ elements, every field $\K$    with characteristic zero, every pair of maps 
$f:[E]^m\rightarrow
\K$,
$g:[E]^n\rightarrow
\K$  such that $fg$ is zero, but $f$ and $g$ are not, then $\tau(supp(f)\cup supp(g))\leq t$. \end{theorem}

With this result, the proof of Theorem \ref{agealgebranozero} is immediate. Indeed,  let $R$ be a relational structure with empty kernel.
 If $\K.\age(R)$, the age algebra of $R$ over $\K$,  is not an integral domain there are two non-zero maps   $f:[E]^m\rightarrow \K$, $f:[E]^n\rightarrow \K$  belonging to $\K.\age(R)$, whose product  $fg$ is  zero. Since $\K$ is an integral domain, none of the integers $m$ and $n$ can be zero.  Since $f$ is $R$-invariant, $m$ is positive  and the kernel $K(R)$  of $R$ is empty, it turns out that $\tau(supp(f))$ is infinite. Hence 
 $\tau(supp(f)\cup supp(g))$ is infinite, contradicting the conclusion of Theorem \ref{main}. 
 
 An other immediate consequence of Theorem \ref{main} is the fact, due to Cameron, that on an infinite set $E$, $e$ is not   a zero-divisor (see Theorem \ref{cameron} below).
\subsubsection{Existence and values of $\tau$} 

The fact the size of  a transversal
can be bounded independently of $f$ and
$g$, and the value of the least upper bound,   seem  to be of independent
interest. 

So, let $\tau(m, n)$ be the least $t$ for which the conclusion of Theorem \ref{main} holds.  

Trivially, we have $\tau(m,n)=\tau(n,m)$. 
We have $\tau(0,n)=\tau(m,0)=0$. Indeed, if $m=0$, $f$ is  defined on the empty set only, an thus  $fg(Q)=f(\emptyset)g(Q)$. Since $\K$ has no non zero divisors, $fg$ is  non zero provided that $f$ and $g$ are non zero. The fact that 
 there is no pair $f,g$ such that $fg$ is zero, but $f$ and $g$ are not, yields $\tau(supp(f)\cup supp(g))=0$. 

We have $\tau(1,n)=2n$ (Theorem \ref {tau(1,n)}). This is a non-trivial fact which essentially amounts to a weighted version of the Gottlieb-Kantor Theorem on incidence matrices (\cite{gottlieb}, \cite{kantor},  see subsection \ref{sub:agealgebra} and Theorem \ref{kantorweight}).
These are the only exact values we know. We prove that $\tau(m, n) $ exists, by supposing that $\tau(m-1,n)$ exists. 
Our existence proof relies in an essential way on Ramsey theorem. It  yields  astronomical upper bounds. For example,  it yields $\tau (2,2)\leq 2(R^2_{k}(4)+2)$ , where $k=5^{30}$ and $R^2_k(4)$ is the Ramsey  number equal to the least integer $p$ such that for every colouring  of the  pairs of $\{1, \dots, p\}$ into $k$ colors there are four integers whose all pairs have the same colour.  The only lower bound we have is  $\tau(2,2)\geq 7$ and more generally $\tau(m,n)\geq (m+1)(n+1)-2$.
 We cannot preclude a
extremely simple upper bound for $\tau(m,n)$, eg quadratic in
$n+m$.

\subsection{Age algebra and profile of a relational structure}\label{sub:agealgebra}
The group agebra was invented  by Cameron in order  to study  the behavior of the function $\theta_G$  which counts for each integer $n$ the number $\theta_G(n)$ of orbits of $n$-subsets of a set $E$ on which acts a permutation group $G$, a function that we call the \emph{orbital profile} of $G$. Groups for which  the orbital profile   takes only finite values are quite important.  Called \emph{oligomorphic groups} by Cameron, they are an objet of study by itself (see Cameron's book\cite{cameronbook}). We present first some properties of the \emph{profile}, a  counting function somewhat more general.  Next, we present the link with the age algebra, then  we gives an  illustration of Theorem \ref{agealgebranozero}. We conclude with some problems. 
\subsubsection{Profile of a relational structure}
The \emph {profile} of a relational structure $R$ with base $E$  is the function
$\profile_R$ which counts for every integer $n$ the number (possibly infinite)
$\profile_R(n)$ of substructures of $R$ induced on the $n$-element
subsets, isomorphic substructures being identified.   Clearly, if $R$ encodes a permutation groups $G$, $\profile_R(n)$ is the number $\theta_G(n)$ of orbits of $n$-element subsets of $E$.   

If the signature $\mu$ is finite (in the sense that $I$ is finite),  there are only finitely many relational structures with signature $\mu$ on an $n$-element domain, hence $\profile_R(n)$ is necessarily an integer for each integer $n$. In order to capture examples coming from algebra and group theory, one cannot preclude $I$ to be infinite. But then, $\varphi_R(n)$ could be an infinite cardinal.  
As far as one is  concerned by the behavior of $\varphi_R$, this case can be  excluded:
\begin{fact}\label{infiniteprofile2}\cite{pouzet06}
Let $n<\vert E\vert $. Then
\begin{equation}\label{infiniteprofile}
 \varphi_R(n)\leq (n+1)\varphi _R(n+1)
\end{equation}
In particular:
\begin{equation}\text{If} \;  \varphi_R(n) \; \text{ is infinite then}\;  \varphi_R(n+1) \; \text{ is infinite  too and }\; \varphi_R(n)\leq \varphi _R(n+1).
\end{equation}
\end{fact}

Inequality (\ref{infiniteprofile})  can be substantially improved:

\begin{theorem}\label{increaseinfinite}
  If $R$ is a relational structure on an infinite set then
  $\profile_R$ is non-decreasing.
\end{theorem}
This result was conjectured with R.Fra{\"\i}ss{\'e} \cite{fraissepouzet1}. We proved it in $1971$; the proof - for a single relation- appeared in 1971 in R.Fra{\"\i}ss{\'e}'s book ~\cite{fraisseclmt1}, Exercise~8
p.~113;  the general case was detailed in  \cite{pouzetrpe}. The proof relies on Ramsey theorem \cite{ramsey}.

More is true:
 \begin{theorem} \label{linalg}If $R$ is a relational structure on a set $E$ having  at least  $2n+m$ elements 
then
$\varphi_{R}(n) \leq \varphi_{R} (n+m)$.\end{theorem}

Meaning that if $\vert E\vert :=\ell$ then  $\varphi_{R}$ increases up to  $\frac{\ell}{2}$; and,
for
$n
\geq
\frac{\ell}{2}$ the value in  $n$ is at least the value of the symmetric of  $n$ w.r.t.
$\frac{\ell}{2}$.

The result is a straightforward consequence of   the following property of incidence matrices.

Let $m,n, \ell$ be three non-negative integers and $E$ be an $\ell$-element set. Let $M_{n, n+m}$ be the matrix whose  rows are indexed by the $n$-element subsets $P$ of $E$ and  columns by the $n+m$-element subsets $Q$ of $E$, the coefficient $a_{P,Q}$ being equal to $1$ if $P\subseteq Q$ and equal to $0$ otherwise. 

\begin{theorem} \label {kantor} If $2n+m\leq l$ then $M_{n, n+m}$ has full row rank (over the field of rational numbers).\end{theorem}

Theorem \ref {kantor} is in W.Kantor 1972 \cite{kantor}, with similar results for affine and vector
subspaces of a vector space.   Over the last 30 years, it as been applied and rediscovered many
times; recently, it was pointed out that it appeared in a 1966  paper of D.H.Gottlieb  \cite{gottlieb}.  Nowadays, this is one of the fundamental tools in algebraic combinatorics. A proof, with a clever argument leading to further developments,  was  given by Fra\"{\i}ss\'e in the $1986$'s  edition of  his book, Theory of relations, see \cite{fraissetr}.

We proved  Theorem \ref{linalg} in 1976 \cite{pouzet76}. The same conclusion was obtained first  for orbits of finite permutation groups  by  Livingstone and  Wagner,  
1965 \cite{livingstone}, and extended   to arbitrary permutation groups by Cameron, 1976 \cite{cameron76}. His proof uses  the dual version of Theorem \ref{kantor}. Later on, he discovered a nice translation in terms  of his age algebra, that we  present now.

For that,  observe that  $\varphi_R$ only depends upon the age of  $R$ and, moreover, if $\varphi_R$ take only integer values, then $\K.\age(R)$ identifies with the set of (finite) linear combinations of members of $\age(R)$. In this case, as pointed out  by Cameron, \emph{$\varphi_R(n)$  is  the dimension of the homogeneous component of degree $n$ of $\K.\age(R)$}. 

Let $e\in \K^{[E]^{<\omega}}$ be the map which is $1$ on the one-element subsets of $E$ and $0$ elsewhere. Let $U$ be the subalgebra  generated by $e$. We can think of $e$ as the sum of isomorphic types of the one-element restrictions of $R$.  
Members of $U$ are then of the form $\lambda_me^m+\cdots+\lambda_1e+\lambda_01$ where $1$ is the isomorphic type of the empty relational structure and $\lambda_m,\dots, \lambda_0$ are in $\K$. Hence $U$ is graded, with $U_n$, the homogeneous component of degree $n$, equals to $\K.e^n$.

Here is the Cameron's result:
 \begin{theorem} \label{cameron}If $R$ is infinite then, for every $u\in \K.\age(R)$, $e u=0$  if and only if $u=0$
 \end{theorem}
 
 This innocent looking result implies that $\varphi_R$ is non decreasing. Indeed,  the image of a basis of $\K.\age(R)_n$ by multiplication by $e^m$ is an independent subset  of  $\K.\age(R)_{n+m}$.   
 
\subsubsection{Growth rate of the profile} Infinite relational structures with a constant profile, equal  to $1$,  were called  {\it monomorphic} and characterized by R. Fra\"{\i}ss\'e who proved that they were {\it  chainable}. Later on, those with bounded profile, called {\it finimorphic}, were characterized as {\it almost chainable} \cite{fraissepouzet1}. Groups with orbital profile equal to $1$ were described by P.Cameron in 1976 \cite{cameron76}. From his characterization, Cameron  obtained that  an orbital profile is ultimately constant, or  grows as fast as a linear function with slope $\frac{1}{2}$.
 
The age algebra can be also used to study the growth of the profile. 

If $A$ is a graded algebra, the \emph{Hilbert function} $h_A$ of $A$ is the function which associates to each integer $n$ the dimension of the  homogeneous component of degree $n$. So, provided that it takes only finite values, the profile $\varphi_R$ is the Hilbert function of the age algebra $\C.\mathcal A(R)$.   In \cite{cameron4},  Cameron made the following important observation about the behavior of the Hilbert fonction.

\begin{theorem}\label {hilbertnonzero} Let  $A$ be a graded algebra over  an algebraically closed field of
characteristic zero.  If
$A$ is an integral domain  the values of the Hilbert function $h_A$ satisfy the inequality 
\begin{equation}\label{eqhilbert}
h_A(n)+h_A(m) -1\leq h_A(n+m)
\end{equation}
for all non-negative integers $n$ and $m$. \end{theorem}

This result has an  immediate consequence on the growth of the profile: 

\begin{theorem}\label{linear} \cite{pouzetrpe}  The growth of the profile of a relational structure with empty kernel  is at least linear provided that it is unbounded.  \end{theorem}

In fact,  
provided that the relational structures satisfy some mild
conditions, the  existence of  jumps in the behavior of the profile extends.

 Let
$\varphi : \N
\rightarrow \N$ and $\psi : \N \rightarrow \N$. Recall that $\varphi = O(\psi)$ and $\psi$
{\it grows as fast as} $\varphi$ if
$\varphi(n) \leq a \psi
(n)$ for some positive real number  $a$ and  $n$ large enough.  We say that
$\varphi$ and $\psi$ have the  
{\it same growth} if $\varphi$ grows as fast as $\psi$ and $\psi$
grows as fast as  $\varphi$. The growth of $\varphi$ is  
{\it polynomial of degree $k$}  if $\varphi$ has the same growth as  $n
\hookrightarrow
n^{k}$; in other words there are positive real numbers $a$ and $b$ such that $an^k\leq \varphi\leq bn^k$ for $n$ large enough. Note that the growth of $\varphi$ is as fast as every polynomial if and only if 
$lim_{n\rightarrow +\infty}\frac{\varphi(n)}{n^{k}}=+\infty$ for every non negative integer $k$.

\begin{theorem}\label{profilpouzet1}
  Let $R := (E, (\rho_i)_{i \in I})$ be a relational structure. The
  growth of $\profile_R$ is either polynomial or as fast as every
  polynomial provided that either the signature $\mu : = (n_i)_{i
    \in I}$ is bounded or the kernel $K(R)$ of $R$ is finite.
\end{theorem}
Theorem  \ref{profilpouzet1} is in \cite{pouzettr}. An outline of the proof  is given in \cite{pouzet06}. A  part  appeared  in \cite{pouzetrpe},  with  a detailed proof showing  that the growth of unbounded profiles of relational structures with bounded signature is at least linear. 

The  kernel of any relational
structure which encodes an oligomorphic  permutation group is finite (indeed,  as already mentionned, if  $R$  encodes a permutation group $G$ acting on a set $E$ then $K(R)$ is the set
union  of the finite orbits of the
one-element subsets of $E$. Since  the number of these orbits is at most
$\theta_G(1)$, $K(R)$ is  finite if $G$ is oligomorphic). Hence: 

\begin{corollary}
The orbital profile of an oligomorphic group is either polynomial or faster than every polynomial.
\end{corollary}

For groups, and graphs, there is a much more precise result than Theorem \ref{profilpouzet1}. It
is due to Macpherson, 1985 \cite{macpherson85}.
 
 \begin{theorem}
The profile of a graph or a permutation groups grows either as a polynomial or as fast as 
$f_{\varepsilon}$, where
$f_{\varepsilon}(n) = e^{n^{\frac{1}{2} - \varepsilon}}$, this for every
$\varepsilon
>0$.
\end{theorem} 

\subsubsection{Growth rate and finite generation} A central question in the study of the profile, raised first by Cameron in the case of oligomorphic groups, is this: 
\begin{problem}\label{growthpoly} If the  profile of a relational structures $R$ with finite kernel has polynomial growth, is $\varphi_R(n)\simeq cn^{k'}$ for some positive real $c$ and some non-negative integer ${k'}$?
\end{problem} 

Let us associate  to a relational structure
$R$ whose  profile takes only finite values its  {\it generating series}
$${\mathcal H}_{\varphi_{R}}:= \sum_{n=0}^{\infty} \varphi_{R}(n)x^{n}$$

\begin{problem}\label{growthseries}
 If $R$  has  a finite kernel and $\varphi_R$ is bounded above by some polynomial, 
is the series
    $\hilbert_{\profile_R}$  a rational fraction
    of the form 
    \begin{equation}\label{quasipoly}
    \frac{P(x)}{(1-x)(1-x^2)\cdots(1-x^k)}\ 
  \end{equation}
 with  $P\in \Z[x]$? 
\end{problem}

Under the hypothesis above we  do not know if   $\hilbert_{\profile_R}$ is a rational fraction. 

It is well known that if a generating function is of the form  $\frac{P(x)}{(1-x)(1-x^2)\cdots(1-x^k)}$ then for $n$ large enough, $a_n$ is a  \emph{quasi-polynomial} of degree $k'$,  with $k'\leq k-1$,  that is a
  polynomial $a_{k'}(n)n^{k'}+\cdots+ a_0(n)$ whose coefficients
  $a_{k'}(n), \dots, a_0(n)$ are periodic functions.  Hence, a subproblem is:

 \begin{problem}\label{growthquasipoly}If $R$  has a finite kernel and $\varphi_R$ is bounded above by some polynomial, 
is  $\profile_R(n) $ a
quasi-polynomial for $n$ large enough?
  \end{problem}
 
  \begin{remark}Since the
  profile is non-decreasing, if $\profile_R(n) $ is a
quasi-polynomial for $n$ large enough then  $a_{k'}(n)$ is eventually
  constant. Hence the profile has polynomial growth in the sense that
  $\profile_R(n)\sim c n^{k'}$ for some positive real $c$ and $k'\in \N$. Thus, in this case,  Problem \ref{growthpoly} has a positive solution.
  \end{remark}
A special case was solved positively with N.Thi\'ery \cite{pouzetthiery}.
 
These problems are linked with the structure of the age algebra. Indeed, 
if  a graded 
 algebra $A$ is finitely generated, then,  since $A$ is a quotient of a polynomial ring $\K[x_{1},\dots, x_{d}]$,  its Hilbert function is bounded above by a polynomial. And, in fact, as it is well known, its Hilbert  series is a fraction of  form $\frac{P(x)}
{(1-x)^d}$, thus of the form given in (\ref{quasipoly}). Moreover, one can choose a numerator  with  non-negative coefficients whenever the algebra is Cohen-Macaulay.  Due to Problem \ref{growthseries}, one could be tempted to conjecture that  these  sufficient   conditions are necessary in the case of age agebras.
Indeed, from Theorem \ref{cameron} one deduces easily:
 \begin{theorem}\label{finitemodule} The profile of $R$ is bounded if and only   if $\K.\age(R)$ is finitely generated as a module over $U$, the graded algebra generated by $e$.  In particular, if one of these equivalent conditions holds,  $\K.\age (R)$ is finitely generated
 \end{theorem} 
 
 But this case is exceptional. The conjecture can be disproved with tournaments. Indeed, on one hand, there are tournaments  whose profile has arbitrarily large polynomial growth rate and, on an other hand,  the age algebra of a tournament  is finitely generated if and only if the profile of the tournament is bounded (this result was obtained with N.Thiery, a proof is presented in \cite{pouzet06}).  
 
 \subsubsection{Initial  segments of an age and ideals of a ring}
No concrete description of  relational structures with  bounded signature, or finite kernel, which have polynomial growth is known. In \cite{pouzettr} (see also \cite {pouzet06}) we proved that if a relational structure $R$ has this property then its age, $\age(R)$,  is {\it well-quasi-ordered} under embeddability, that is  every final segment of $\age (R)$ is finitely generated, which amounts to the fact that the collection  $F(\age(R))$ of final segments of $\age(R)$ is noetherian, w.r.t. the inclusion order. Since the   fundamental paper of Higman\cite {higman52},  applications of the notion of  well-quasi-ordering have proliferated (eg see the Robertson-Seymour's theorem for an application to graph theory \cite{diestel} ). Final segments play for  posets the same role than ideals for rings. Noticing that  an age algebra is finitely generated if and only if it is noetherian,  we are lead to  have a closer look at the relationship between the basic objects of the theory of relations and of ring theory, particularly ages and ideals.   

We mention the following result which will be incorporated into a joint paper with  N.Thi\'ery. 
\begin{proposition} Let $\mathcal A$ be the  age of a relational structure $R$ such that the profile of $R$  takes only finite values and $\K.\age $ be its age algebra. If $\mathcal A'$ is an initial segment of $\mathcal A$ then:
 \begin{enumerate} [{(i)}]
 \item The vector subspace $J:= \K.(\mathcal A\setminus \mathcal A')$ spanned by $\mathcal A\setminus \mathcal A'$ is an ideal of $\K.\age $. Moreover, the quotient of $\K. \age$ by $J$ is  a ring isomorphic to  the ring $\K.\mathcal A'$. 
 \item If this ideal is irreducible then $\mathcal A'$ is a subage of $\mathcal A$.
  \item This is a prime ideal if and only if $\mathcal A'$ is an inexhaustible age.
 \end{enumerate}
\end{proposition}
The proof of Item $(i)$ and Item $(ii)$ are immediate. The proof of Item $(iii)$ is essentially based on Theorem \ref{agealgebranozero}. 

According to Item $(i)$, $F(\mathcal A)$ embeds into the collection of ideals of $\K.\mathcal A)$. Consequently:
\begin{corollary} If an age algebra is finitely generated then the age is well-quasi-ordered by embeddability.
\end {corollary}

\begin{problem}
How the finite generation of an age algebra translates in terms of embeddability between members of the ages?
\end{problem}

\subsubsection{Links with language theory}

In the  theory of  languages,  one of the basic results is that the generating series of a
regular language is a rational fraction (see \cite{berstel}). This result is not far away from our considerations. Indeed, if  $\mathcal A$ is  a finite alphabet, with say $k$ elements,  and $\mathcal A^*$ is the set of words over $\mathcal A$, then each word can be viewed as a finite chain coloured by $k$ colors. Hence $\mathcal A^*$ can be viewed as  the age of the  relational structure  $R$ made of the chain $\Q$ of rational numbers  divided into $k$ colors 
in such a way that,  between two distinct rational numbers, all colors  appear. Moreover, as pointed out by Cameron \cite{cameron1}, the age algebra $\Q.\age(R)$ is isomorphic to the shuffle algebra over $\mathcal A$, an important object in algebraic combinatorics (see \cite {lothaire}).
\begin{problem} 
Does the members of the age of a relational structure with polynomial growth can be coded by  words forming a regular language?
\end{problem}

\begin{problem} Extend the properties of  regular languages to subsets of the collection $\Omega_{\mu}$ made of isomorphic types of finite relational structures with signature $\mu$.
\end{problem}


\section{Proof of Theorem \ref{main}}\label {sec:proof} The proof idea of Theorem \ref{main} is very simple and we give it first. 

We prove the result by induction. We suppose that it holds for pairs  $(m-1, n)$.  Now, let  
$f:[E]^m\rightarrow
\K$ and 
$g:[E]^n\rightarrow
\K$  such that $fg$ is zero, but $f$ and $g$ are not. As already mentionned,  $m$ and $n$ are non zero, hence members of $supp(f)\cup supp(g)$ are non empty. Let $sup(f,g):= \{ (A,B)\in supp(f)\times supp(g):   A\cap B=\emptyset\}$.  We may suppose $sup(f,g)\not =\emptyset$, otherwise the conclusion of Theorem \ref{main} holds with  $t:=m+n-1$. For the sake of simplicity, we suppose that $\K:=\Q$. In this case, we color elements $A$ of $[E]^m$ into three colors :-,0, +, according to the value of $f(A)$. We do the same with elements $B$ of $[E]^{n}$ and we color each member $(A,B)$ of $supp(f,g)$ with the colors of its components. With the help of  Ramsey' theorem and a lexicographical ordering, we prove that if the transversality is large enough there is an $(m+n)$-element subset $Q$ such that all pair $(A,B)\in supp(f,g)(Q):= supp(f,g)\cap ([Q]^{m}\times [Q]^{n})$ have the same color. This readily implies that $fg(Q)\not =0$, a contradiction. If $\K\not =\Q$, we may replace the three colors by five, as the following lemma indicates.

\begin{lemma}\label{lem:field}Let $\K$ be a field with characteristic zero. There is a partition of $\K^*:= \K\setminus \{0\}$ into at most
four blocks such that for every integer 
$k$  and every $k$-element sequences
$(\alpha_1,\dots,\alpha_k)\in D^k$ ,
$(\beta_1,\dots,\beta_k)\in D^{'k}$, where $D$, $D'$ are two blocks of the partition of $\K^*$, then
$\sum_{i=1}^{k}\alpha_{i}\beta_{i}\in \K^*$.
\end{lemma}
\begin {proof}
This holds trivially if $\K:=\C$. For an example, divide $\C^*$ into the  sets $D_{i}:= \{z\in
\C^*:
\frac{\pi i}{2}\leq Arg z<
\frac{\pi (i+1)}{2}\}$ ($i<4$). If $\K$ is arbitrary, use the Compactness theorem of first-order logic, under the form of the "diagram
method" of A.Robinson \cite{hodges}. Namely, to the language of fields, add names for the elements of $\K$, a binary
predicate symbol, and axioms, this  in such a way that a model, if any, of the resulting theory $T$ will be an extension of
$\K$ with a partition satisfying the conclusion of the lemma. According to the Compactness theorem of first-order logic, the existence of a model of $T$, alias the consistency of $T$, reduces to the consistency of
every  finite subset $A$ of $T$. A finite subset $A$ of $T$ leads to a finitely generated subfield of $\K$. Such subfield is isomorphic to a
subfield  of
$\C$ (see \cite {hodges} Example 2, p.99, or \cite{bourbaki} Proposition 1, p. 108). This latter  subfield equipped with the partition induced by the partition existing on $\C^*$  satisfies the conclusion of the lemma, hence is a model of $T$, proving that 
$A$ is consistent.  \end{proof}\\

Let $\T^*$ be the set of these four blocks, let $\ T:= \T\cup \{0\}$ and let $\chi $ be the map from $\K$ onto $\T$. 

%


  \subsection{Invariant relational structures and their age algebra}\label {sec:proof}
 \subsubsection{Isomorphism, local isomorphism}
  Let $R:= (E, (\rho_i)_{i\in I})$ and $R':= (E', (\rho'_i)_{i\in I})$ be two
relational structures having the same signature $\mu:=(m_i)_{i\in I}$.
 A map $h: E\rightarrow E'$ is \emph{an isomorphism  from} $R$ {onto} $R'$ if
 \begin{enumerate}
 \item $h$ is bijective,
 \item $(x_1, \dots, x_{m_i})\in \rho_{i}$ if and only if $(h(x_1), \dots, h(x_{m_i}))\in \rho'_{i}$ for every $(x_1, \dots, x_{m_i})\in E^{m_i}$, $i\in I$.
 \end{enumerate}
A \emph{ partial map of $E$} is a map $h$ from a subset $A$ of $E$ onto a subset $A'$ of $E$, these subsets are the \emph{domain} and \emph{codomain} of $h$.  A
\emph{local isomorphism}  of $R$ if  a partial map $h$ which is an  isomorphism from $R_{\restriction A}$ onto $R_{\restriction A'}$ (where $A$ and $A'$ are the domain an codomain of $h$). 

\subsubsection{Invariant relational structures}
A {\it chain} is a pair $L:= (C, \leq )$  where $\leq $ is a linear order on $C$.  Let $L$ be  a chain. Let $V$ be a non-empty set, $F$ be a set disjoint from $V\times C$ and let $E:=F\cup (V\times C)$.  Let $R$ be a relational structure   with base set $E$. Let $r$ be a non-negative integer, $r\leq \vert C\vert$. Let $X,X'\in [C]^{r}$. Let $\ell$ be the unique order isomorphism from $L_{\restriction X}$ onto $L_{\restriction X'}$ and  let $\overline \ell:= 1_F\cup (1_V, \ell)$ be the  partial map such that $\overline \ell(x)=x$ for $x\in F$ and $\overline \ell(x,y)=(x, \ell(y))$ for $(x,y)\in V\times X$.  

We say that $X$ and $X'$ are \emph{equivalent} if $\overline \ell$ is  an isomorphism of $\mathcal H_{\restriction F\cup V\times X}$ onto $\mathcal H_{\restriction F\cup V\times X'}$. This defines an equivalence relation on $[C]^{r}$. 

We say that $R$ is \emph{$r-F-L$-invariant} if two arbitrary members of $[C]^r$ are equivalent. We say that $R$ is \emph{$F-L$-invariant} if it is $r-F-L$-invariant for every non-negative integer $r$, $r\leq \vert C\vert$.

It is easy to see that if  the signature $\mu$ of $R$ is bounded and $r:=Max( \{m_i: i\in I\})$, $R$ is  $F-L$-invariant  if and only if it is $r'-F-L$-invariant  for every $r'\leq r$. In fact:

\begin{lemma} \label{lem:invar} If $\vert C\vert>r:=Max( \{m_i: i\in I\})$, $R$ is  $F-L$-invariant  if and only if it is $r-F-L$-invariant. 
\end{lemma}
This is an immediate consequence  of the following lemma:

 \begin{lemma}\label{heredit}
If  $R$ is $r-F-L$-invariant and $r<\vert C\vert$ then $R$ is $r'-F-L$-invariant for all $r'\leq r$. 
\end{lemma}
\begin{proof} We only prove that $R$ is $(r-1)-F-L$-invariant. This suffices. Let $X,X'\in [C]^{r-1}$. Since $r<\vert C\vert$, we may select $Z\in [C]^r$ such that the last element of $Z$ (w.r.t. the order $L$) is strictly below some  element $c\in C$. 
\begin{claim} \label{claim:triv1} There are $Y,Y'\in [Z]^{r-1}$ which are equivalent to $X$ and $X'$ respectively.\end{claim}

{\bf Proof of Claim \ref{claim:triv1}.}
Extend $X$ and $X'$ to two $r$-element subsets $X_1$ and $X'_1$ of $C$. Since $R$ is $r-F-L$-invariant, $X_1$ is equivalent to $Z$, hence the unique isomorphism from $L_{\restriction X_1}$ onto $L_{\restriction Z}$ carries $X$ onto an equivalent  subset $Y$ of $Z$. By the same token, $X'_1$ is equivalent to a subset $Y'$ of $Z$. 
\endproof

\begin{claim} \label{claim:triv2}
$Y$ and $Y'$ are equivalent.
\end{claim}

{\bf Proof of Claim \ref{claim:triv2}.} 
The unique isomorphism from $L_{\restriction Y\cup\{c\}}$ onto $L_{\restriction Y'\cup\{c\}}$  carries $Y$ onto $Y$ hence $T$ and $Y'$ are equivalent.
\endproof

From the two claims above 
$X$ and $X'$ are equivalent. Hence, $R$ is $(r-1)-F-L$-invariant. 
\end{proof}

\subsubsection{Coding  by words}
Let $\mathcal A:={\mathfrak  P}(V)\setminus \{\emptyset\}$. 
Let $\mathcal {A}^*: = \bigcup_{p<\omega} \mathcal {A}^{p}$ be the set of finite sequences of members of $\mathcal A$. A finite sequence $u$ being viewed as a word on the 
alphabet $\mathcal A$,  we write it as a juxtaposition of letters and we denote by  $\lambda$ the empty sequence;  the \emph {length} of $u$, denoted by $\vert u\vert $  is the number of its terms.  Let $p$ be a non negative integer. If $X$ is a subset of $p:=\{0,\dots, p-1\}$ and $u$ a word of length $p$, the restriction of $u$ to $X$ induces a word that we denote by $t(u_{\restriction X})$. We suppose that $V$ is finite and  we equip  $\mathcal {A}$ with a linear order. We compare words with the same length with the lexicographical order, denoted by $\leq_{lex}$. We record without proof the following result.
 \begin{lemma}\label{increasing} Let $p,q$ be two non negative integers and $X$ be an $p$-element subset of $p+q:=\{0, \dots, p+q-1\}$. The map from $\mathcal A^p\times \mathcal A^q$ into $\mathcal A^{p+q}$ which associates to every pair $(u,v)\in \mathcal A^p\times \mathcal A^q$  the unique word $w\in \mathcal A^{p+q}$ such that $t(w_{\restriction X})=u$ and $t(w_{\restriction p+q\setminus X})=v$ is strictly increasing (w.r.t. the lexicographical order).
\end{lemma}
This word $w$ is a \emph{shuffle} of $u$ and $v$ that we denote $u_X\shuffle v$. We denote by $u\widehat\shuffle v$ the largest word of the form $u_X\shuffle v$.

We order  $\mathcal{A}^*$ with the \emph {radix order} defined as follows:
if  $u$ and  $v$ are two distincts words, we set  $u<v$ if and only if either $\vert u \vert <\vert v\vert $ or  $\vert u \vert =\vert v\vert $ et $u<_{lex}v$. 
We suppose that $F$ is finite and we order $\mathfrak {P}(F)$ in such a way that  $X<Y$ implies $\vert X\vert \geq \vert Y\vert$. 
Finally, we order $\mathfrak {P}(F)\times \mathcal{A^*}$ lexicographically. 

Let $L:= (C, \leq)$. Let $Q$ be a finite subset of $E:= F\cup (V\times C)$.  Let $proj(Q):=\{ i \in C: Q\cap V\times \{i\}\not = \emptyset \}$.  Let $i_0,Ê\dots, i_{p-1}$ be an enumeration of $proj(Q)$ in an increasing order (w.r.t $L$) and let  $w(Q\setminus F)$ be the  word $u_{0}\dots u_{p-1}\in \mathcal A^*$ such that $Q\setminus F=u_0\times \{i_0\}\cup \cdots \cup u_{p-1}\times\{i_{p-1}\}$. We set  $\overline w(Q):= (Q\cap F, w(Q\setminus F))$. If $\mathcal Q$ is a subset of $[E]^{<\omega}$,  we set  $\overline w(\mathcal Q):=\{ \overline w(Q): Q\in \mathcal Q\}$. If $f: [E]^m\rightarrow \K$, let  $lead(f):=-\infty$ if $f=0$ and otherwise let $lead (f)$ be the largest element of $\overline w(supp(f))$.
We show below that this latter parameter behaves as the degree of a polynomial. 

We start with an easy fact. 
\begin{lemma}\label{littletrick} Let  $m$ and $n$ be two non negative integers, $A\in [E\setminus F]^m$ and $B\in [E]^n$. If $\vert C\vert \geq m+n$ there is $A'\in [E\setminus F]^m$ such that $proj(A')\cap proj(B)=\emptyset $ and $\overline w(A')=\overline w(A)$.\end{lemma}
\begin{lemma}\label{lem:good1}
Let $R$ be an $F-L$-invariant structure on $E$. Let $m$ and $n$ be two non negative integers; let $f: [E]^m\rightarrow \K$, $g:[E]^n\rightarrow \K$ be two non zero members of $\K.\age (R)$. Let $A_0\in supp(f)$, and $B_0\in supp(g)$ such that $\overline w(A_0)=lead(f)$ and $\overline w(B_0)=lead(g)$.
Suppose that $F$ and $V$ are  finite, that $\vert C\vert \leq n+m$ and $supp(f)\cap [E\setminus F]^m\not =\emptyset$. Then: 

\begin{equation}\label{eq:leading0}
supp(f,g)\not = \emptyset.
\end{equation}
 \begin{equation}\label{eq:leading1} (\overline w(A), \overline w(B))=(lead(f), lead(g)).
 \end{equation}
 for all $(A,B)\in supp(f,g)(Q_0)$, where $\overline w(Q_0)=lead(f, g)$ and $lead(f,g)$ is the largest element of $\overline w(\{ A\cup B: (A,B)\in supp(f,g)\})$.
  
\begin{equation}\label{eq:leading2} (f(A), g(B))=(f(A_0),g(B_0))\end{equation} for every $(A,B)\in supp(f,g)(Q_0)$.

\begin{equation}\label{eq:leading3} fg(Q_0)=\vert supp(f,g)(Q_0)\vert f(A_0)g(B_0).\end{equation}

 \begin{equation}\label{eq:leading4}lead(fg)=lead(f,g)=(Q_0\cap F, w(A_0)\widehat\shuffle w(B_0\setminus F)).\end{equation}

\end{lemma}
\begin{proof}
\begin{enumerate}
\item {\bf Proof of (\ref{eq:leading0}).}
Since $supp(f)\cap [E\setminus F]^m\not =\emptyset$ and the order on $\mathfrak P(F)$ decreases with the size, $A_0$ is disjoint from $F$. Let $B\in supp(g)$. According to  Lemma \ref{littletrick} there is $A'$ such that $proj(A')\cap proj(B)=\emptyset $ and $\overline w(A')=\overline w(A_0)$. We have $A'\cap B=\emptyset$ and, since $f\in \K.\age(R)$ and $R$ is $F-L$-invariant, $f(A')=f(A_0)$. Thus  $(A', B)\in supp(f,g)$.\endproof

\item {\bf Proof of (\ref{eq:leading1}).}
Let $(A,B)\in supp(f,g)(Q_0)$. Since $A\in supp(f)$ and $B\in supp(g)$, we have trivially:
\begin{equation}\label{eq:leading5} \overline w(A)\leq lead(f)\; \text{and} \; \overline w(B)\leq lead(g).
\end{equation}

\begin{claim} $B\cap F=Q_0 \cap F$ and  $A\cap F=\emptyset$. 
\end{claim}  The pair $(A', B)$ obtained in the  proof of (\ref{eq:leading0}) belongs to  $supp(f,g)$. Let $Q':= A'\cup B$. By maximality of $\overline w(Q_0)$, we have $\overline w(Q')\leq \overline w(Q_0)$. If $B\cap F\not =Q_0 \cap F$, then $\vert Q'\cap F\vert <\vert Q_0\cap F\vert $, hence $\overline w(Q_0)<\overline w(Q')$. A contradiction. The fact that $A\cap F=\emptyset$ follows.\endproof

\begin{claim}
$proj(A)\cap proj(B)=\emptyset$ and $\vert proj(A)\vert =\vert proj(B).\vert $\end{claim}

Apply Lemma \ref{littletrick}. Let  $A'$ such that $proj(A')\cap proj(B)=\emptyset $ and $\overline w(A')=\overline w(A)$. Since $f\in \K.\age(R)$ and $R$ is $F-L$-invariant, $f(A')=f(A)$ thus  $(A', B)\in supp(f,g)$. Set $Q':= A'\cup B$. We have $w(Q'\setminus F)\leq w(Q_0\setminus F)$ hence $\vert w(Q')\vert \leq \vert w(Q_0)\vert$. Since $\vert w(Q'\setminus F)\vert=\vert proj(A')\vert +\vert proj(B)\vert$ and $\vert w(Q_0\setminus F)\vert \leq \vert proj(A)\vert+\vert proj(B)\vert $, we get $\vert w(Q_0\setminus F)\vert =\vert proj(A)\vert+\vert proj(B)\vert $. This  proves our claim.\endproof

Let $i_0,\dots, i_{r-1}$ be an enumeration of  $proj (Q_0\setminus F)$ in an increasing order. Let $X:= \{j\in r: i_j\in proj(A)\}$. Since $proj(A)\cap proj(B)=\emptyset$, we have $w(Q_0\setminus F)=w(A)_X\shuffle w (B\setminus F)$. Since $w(A)\leq w(A_0)$ and $\vert w(A)\vert = \vert w(A_0)\vert$, Lemma \ref{increasing} yields $w(Q_0\setminus F)\leq w(A_0)_X\shuffle w(B)$. As it is easy to see, there is $A'_0$ such that $\overline w(A'_0)=\overline w(A_0)$ and $Q':= A'_0\cup B$ satisfies $w(Q'\setminus F)= w(A_0)_X\shuffle w(B)$. Since $(A', B)\in supp(f,g)$, we have $w(Q')=w(Q_0)$ by maximality of $w(Q_0)$. With Lemma \ref{increasing} again,  this yields $w(A)=w(A_0)$. Hence  $\overline w(A)=\overline w(A_0)$. A similar argument yields $w(B\setminus F)= w(B_0\setminus F)$ and also $\overline w(B\setminus F)= \overline w(B_0\setminus F)$.\endproof

 \item {\bf Proof of (\ref{eq:leading2}).} Since $R$ is $F-L$-invariant, from $\overline w(A)=lead (f):=\overline w(A_0)$ we get $f(A)=f(A_0)$. By the same token, we get $g(B)=g(B_0)$.\endproof 
\item {\bf Proof of (\ref{eq:leading3}).}
Since $fg(Q_0)=\sum_{(A,B)\in sup(f,g)}f(A)g(B)$ the result follows from (\ref{eq:leading2}).\endproof
\item {\bf Proof of (\ref{eq:leading4}).}
 From (\ref{eq:leading3}),  $fg(Q_0)\not =0$, the equality $lead(fg)=lead(f,g)$ follows. The remaining equality follows from (\ref{eq:leading1}).\endproof
\end{enumerate}
With this, the proof of the lemma is complete.

\end{proof}

As far as invariant structures are concerned, we can retain this:
\begin{corollary} Under the hypotheses of Lemma \ref{lem:good},  $fg\not= 0$.
\end{corollary}


\subsubsection{An application}
Let $m, n$ be two positive integers, $E$ be a set and
$f:[E]^m\rightarrow
\K$,  $g:[E]^n\rightarrow
\K$. Let  $R:= (E, (\rho_{(i,j)})_{(i,j)\in \T^{*}\times2)})$ be the relational structure  made of the four $m$-ary relations $\rho_{(i,0)}:=\{(x_1, \dots, x_m):\chi\circ f(\{x_1,\dots, x_m\})=i\}$ and the four  $n$-ary relations $\rho_{(i,1)}:=\{(x_1, \dots, x_n): \chi\circ g(\{x_1,\dots, x_n\})=i\}$.  A map $h$ from a subset $A$ of $E$ onto a subset $A'$ of $E$ is a
\emph{local isomorphism}  of $R$ if $\chi\circ f (P)=\chi\circ f (h[P])$ and $\chi\circ f (R)=\chi\circ f (h[R])$ for every $P\in [A]^m$, every $R\in [A]^n$. This fact allows us to consider  the pair $\mathcal H:= (E, (\chi\circ f, \chi\circ g))$ as  a relational structure.
 In the sequel we suppose that  $E=F\cup (V\times C)$ with $F$ and $V$ finite; we fix a chain $L:= (C, \leq)$. 

 \begin{lemma}\label{lem:good}
Suppose that there are $P\in supp(f)\cap [V\times C]^m$ and $R\in supp(g)\cap [E\setminus P]^n$. If $\mathcal H$ is $F-L$-invariant and $\vert C\vert \geq m+n$. Then $fg\not =0$. 
\end{lemma}

\begin{proof}  Let $lead(f,g)$ be the largest element of $\overline w(\{ A\cup B: (A,B)\in supp(f,g)\})$ and let $Q_0$ such that  $\overline w(Q_0)=lead(f, g)$. 
\begin{claim} \label {claim:key}$(\chi\circ f(A), \chi\circ g(B))$ is constant for $(A,B)\in sup(f,g)(Q_0)$. 
\end{claim}
{\bf Proof of Claim \ref{claim:key}}. Let $s: \T \rightarrow \K$ be a section of $\chi$. Let $f':=s\circ \chi\circ f$  and let $g':=s\circ \chi\circ g$. Then $f',g'\in \K.\age (\mathcal H)$ and  $supp(f',g')=supp (f,g)$.
According to Equation (\ref{eq:leading2}) of Lemma \ref{lem:good1},  $(f'(A), g'(B))$ is constant for $(A,B)\in supp(f',g')(Q_0)$. The result follows.\endproof

 From  Lemma \ref{lem:field}, $fg(Q_0):=\sum_{(A,B)\in sup(f,g)(Q_0)}f(A)g(B)\not =0$.
\end{proof}

 We recall the finite version of the  theorem of Ramsey \cite{ramsey}, \cite{graham}.
 \begin{theorem}\label{ramsey}For every integers $r,k,l$ there is an integer $R$ such that  for every partition of the $r$-element subsets of a $R$-element set $C$  
into $k$ colors there is  a
$l$-element subset $C'$ of $C$  whose all $r$-element subsets have the same color. 

\end{theorem}
The least integer $R$ for which the conclusion of Theorem \ref{ramsey} holds is a Ramsey number that we denote $R_{k}^{r}(l)$.

Let $m$ and $n$ be two non negative integers. Set $r:=Max(\{m,n\})$,   $s:= {mr +n\choose
m}+{mr+n\choose n}$, $k:=5^{s}$ and, for an  integer $l$, $l>r$,  set  $\nu(l):=R_{k}^{r}(l)$.

\begin{lemma} \label{lem:ram}If $\vert F\vert =n$, $\vert V\vert=m$ and  $\vert C\vert \geq \nu(l)$ there is an $l$-element subset $C'$ of $C$ such that $\mathcal H_{\restriction F\cup V\times C'}$ is $F-L_{\restriction C'}$-invariant.
\end{lemma}
 The proof is a  basic application of Ramsey theorem. We give it for reader convenience. See \cite {fraissetr} 10.9.4 page 296, or \cite {pouzetri} Lemme IV.3.1.1 for a similar result).

\begin{proof} The number of equivalence classes on $[C]^r$ is at most $k:=5^{s}$ (indeed, this number is bounded by the  number of distinct pairs $(\chi\circ f',\chi\circ g')$ such that  $f'\in \K^{[E']^m}$, $g'\in \K^{[E']^n}$ and $\vert E'\vert=n+mr$). Thus,  according to Theorem \ref{ramsey}, there is a  $l$-element subset $C'$ of $C$  whose all $r$-element subsets are equivalent. This means that $\mathcal H_{\restriction F\cup V\times C'}$ is $r-F-L_{\restriction C'}$-invariant. Now, since $r<l$, Lemma \ref{heredit} asserts that $\mathcal H_{\restriction F\cup V\times C'}$ is $r'-F-L_{\restriction C'}$-invariant for all $r'\leq r$. Since the signature of $\mathcal H$ is bounded by $r$,  $\mathcal H_{\restriction F\cup V\times C'}$ is $F-L_{\restriction C'}$-invariant from Lemma \ref{lem:invar}.
\end{proof}

\subsection{The existence of $\tau(m,n)$}
Let $m$ and $n$ be two non negative integers. Suppose $1\leq m\leq n$ and that $\tau(m-1,n)$ exists. Let $E$ be  a set with at least $m+n$ elements and 
$f:[E]^m\rightarrow
\K$,   
$g:[E]^n\rightarrow
\K$  such that $fg$ is zero, but $f$ and $g$ are not.

\begin{lemma} \label{discharging}Let $A$ be  a transversal for $supp(f)$. Then there is a transversal $B$ for $supp(g)$ such that:
\begin{equation}\label{eq:transv}
\vert B\setminus A\vert \leq \tau(m-1,n)
\end{equation}
\end{lemma}
\begin{proof}   Among the sets $P\in supp(f)$ select one,  say $P_0$, such that $F_0:= P_0\cap A$ has minimum size, say $r_0$. 

{\bf Case 1.}  $A \cup P_0$ is also a transversal for $supp(g)$.
In  this case,  set $B:= A\cup P_0$.  We have   $\vert B\setminus A\vert \leq m-1$,  hence inequality (\ref{eq:transv}) follows  from inequality (\ref{eq:transv2}) below:
\begin{equation}\label{eq:transv2}
m-1\leq \tau(m-1,n)\end{equation}

This inequality is trivial for $m=1$. Let us prove it for  $m>1$ (a much better inequality  is given in Lemma \ref{lem:minor}). Let $E'$ be an $m+n-1$-element  set, let $f':[E']^{m-1}\rightarrow \K$ and  $g':[E']^{n}\rightarrow \K$, with $f'$ non zero on a single $m-1$-element set $A'$,  $g'(B')=1$ if  $A'\cap B'\not= \emptyset$ and $g'(B')=0$ otherwise. Since $m-1$ and $n$ are not $0$, $f'$ and $g'$ are not $0$. Trivially, $f'g'=0$ and,  as it easy to check, $A'$ is a transversal for $supp(f')\cup supp(g')$ having minimum size, hence   $\tau(supp(f')\cup supp(g'))=m-1$. 

{\bf Case 2.} Case 1 does not hold.  In this case, pick $x_0\in F_0$ and set  $E':= (E\setminus A) \cup (F_0\setminus \{x_0\})$. Let $f':  [E']^{m-1}\rightarrow \K$ be defined by setting $f(P'):= f(P'\cup \{x_0\})$ for all $P'\in [E']^{m-1}$ and let $g':=g_{\restriction [E']^{n}}$.

%
\begin{claim} \label{claim:zero} $\vert E'\vert \geq m+n-1$ and $f'g'=0$.
\end{claim}

\noindent{\bf Proof of Claim \ref{claim:zero}.}
The inequality follows from the fact that $A \cup P_0$ is not a transversal for $supp(g)$. Now, let
 $Q'\in [E']^{m+n-1}$ and $Q:= Q'\cup \{x_0\}$. 

\begin{equation}\label{eq:1}f'g'(Q')= \sum_{P'\in [Q']^{m-1}}  f'(P')g'(Q'\setminus P')= \sum_{x_0\in P\in [Q]^m}  f(P)g(Q\setminus P)\end{equation} Since $fg=0$ we have:
\begin{equation}\label{eq:2} 0=fg(Q)=\sum_{P\in [Q]^{m}}  f(P)g(Q\setminus P)=\sum_{x_0\in P\in [Q]^{m}}  f(P)g(Q\setminus P)+\sum_{x_0\not \in P\in [Q]^{m}}  f(P)g(Q\setminus P)\end{equation}
If $x_0\not \in P$, $\vert P\cap  A\vert <r_0$,  hence $f(P)=0$. This implies that the second term in the last member of (\ref{eq:2}) is zero, hence the second member of (\ref{eq:1}) is zero. This proves our claim.\endproof  

From our hypothesis $A\cup P_{0}$ is not a tranversal for $g$. Hence, we have\begin{claim} \label{claim:bothnonzero}$f'$ and $g'$ are not zero. 
\end{claim}

The existence of $\tau(m-1, n)$ insures that there is a transversal $H$ for $supp(f')\cup supp(g')$ of size at most $\tau(m-1,n)$. The set $B:=A\cup H$ is a transversal for $supp(f)\cup supp(g)$.

\end{proof}

\begin{lemma}\label{lem:bound1}
Let $l$ be  a positive integer. If $\tau(supp(f)\cup supp(g))> n+m(l-1)+ \tau(m-1,n)$ then for every   $F\in supp(g)$ there is a subset  $\mathcal P\subseteq  supp(f)\cap [E\setminus F]^m$ made of  at least $l$ pairwise disjoint sets.
\end{lemma}
\begin{proof}
Fix $F\in supp(g)$.  Let  $\mathcal P\subseteq  supp(f)\cap [E\setminus F]^m$ be a finite  subset  made of   pairwise disjoint sets and let $p:= \vert \mathcal P\vert$. If the conclusion of the lemma does not hold, we have $p<l$.  Select then $\mathcal P$ with maximum size and set $A:=F\cup \bigcup \mathcal P$.  Clearly  $A$ is a transversal   for $supp(f)$.  According to Lemma \ref{discharging} above,  $\tau(supp(f)\cup supp(g))\leq \vert A\vert +\tau(m-1, n)=n+mp+ \tau(m-1,n)$. Thus, according to our hypothese, $l\leq p$. A contradiction.  \end{proof}

Let $\varphi(m,n):=n+ m(\nu(n+m)-1)+\tau(m-1,n)$. 

\begin{lemma}\label{lem:tau}
\begin{equation}\label{eq:tau}
\tau(m,n)\leq \varphi (m, m)
\end{equation}
\end{lemma}
\begin{proof}
Suppose $\tau(supp(f)\cup supp(g))>\varphi(m,n)$. Let $F\in supp(g)$. According to Lemma \ref{lem:bound1} there is a subset $\mathcal P\subseteq  supp(f)\cap [E\setminus F]^m$ made of  at least $\nu(n+m)$ pairwise disjoint sets. With no loss of generality, we may suppose that $\cup \mathcal P$ is a set of the form $V\times C$ where $\vert V\vert=m$ and $\vert C\vert=\nu(m,n)$. Let $\leq$ be a linear order on $C$ and $L:= (C, \leq)$. According to Lemma \ref{lem:ram} there is an $n+m$-element subset $C'$ of $C$ such that $\mathcal H_{\restriction F\cup V\times C'}$ is $F-L_{\restriction C'}$-invariant. According to Lemma \ref{lem:good} $fg\not =0$. A contradiction. \end{proof}

With Lemma \ref{lem:tau}, the proof of Theorem \ref{main} is complete.

Note that  $\varphi(1,2)=1+R^2_{5^{10}}(3)+\tau(0,n)=R^2_{5^{10}}(3)$, whereas $\tau(1,2)=4$. Also $\varphi (2,2)=2R^2_{5^{30}}(4)+\tau(1,2)= 2(R^2_{5^{30}}(4)+2)$.  

Our original  proof of Theorem \ref{main} was a bit simpler. Instead of an $m$-element set $F$ and several pairwise disjoint $n$-element sets, we considered several pairwise $n+m$-element sets. In the particular case of $m=n=2$,  we got  $\tau (2,2)\leq 4R_{5^{56}}^2(4)+1$. In term of  concrete upper-bounds, we are not convinced that the  improvement worth the effort.

\section{The Gottlieb-Kantor theorem and the case $m=1$ } 
Let $E$ be  a set. To each $x\in E$ associate an  indeterminate $X_{x}$.  Let $\K[E]$ be the algebra over 
 the field $\K$ of polynomials in these indeterminates. Let  $f:[E]^{1} \rightarrow \K $. Let $D_{f}$ be the derivation on this algebra which is induced by $f$, that is  $D_{f}(X_{x}):=f(\{x\})$ for every $x\in E$. Let  $\varphi: \K[E]\rightarrow \K[E]$ be the ring homomorphism such that 
$\varphi_f(1):=1$ and  $\varphi_f(X_{x}):= f(\{x\}) X_{x}$. Let $e: [E]^{1} \rightarrow K$ be the constant map equal to $1$ and 
let $D_{e}$ be the corresponding derivation. For example $D_f(X_xX_yX_z)=f(\{x\})X_yX_z+f(\{y\})X_xX_z+f(\{z\})X_xX_y$ whereas $D_e(X_xX_yX_z)=X_yX_z+X_xX_z+ X_xX_y$. It is easy to check that:
\begin{equation} \label{eq:derivation}D_{e}\circ \varphi_f=\varphi_f\circ D_{f}. 
\end{equation}
Let $n$ be a	non negative integer; let $\K_{[n]}[E]$ be the vector space generated by the monomials made of $n$ distinct variables.
From equation (\ref{eq:derivation}), we deduce:
\begin{corollary} \label{cor:surj}If  $f$ does not take the value zero on $[E]^1$, 
the surjectivity of  the maps from $\K_{[n+1]}[E]$ into $\K_{[n]}[E]$ induced by $D_{f}$  and $D_{e}$  are equivalent. 
\end{corollary}
Suppose that $E$ is finite. In this case,  the matrix of the restriction of $D_e$ to $\K_{[n+1]}[E]$ identifies to   $M_{n,n+1}$. Thus, according to Theorem \ref{kantor}, $D_{e}$ is surjective provided that $\vert E\vert \geq 2n+1$. Corollary \label{cor:surj} asserts that in this case, $D_f$ is surjective too.  This yields:
\begin {lemma} \label{lem:weak}If $f$ does not take the value zero on a subset  $E'$ of $E$ of size at least  $2n+1$ but $f g=0$ for some $g:[E]^n\rightarrow K$  then $g$ is zero on the  $n$-element subsets of $E'$.
\end{lemma} 
\begin{proof}
Suppose that  $fg=0$. Let $\overline g:\K_{[n]}[E]\rightarrow \K$ be the linear form defined by setting $\overline g(\Pi_{x\in B} X_x):= g(B)$ for each $B\in [E]^{n}$. Then  $\overline g\circ D_f$ is $0$ on $\K_{[n+1]}[E]$. From  Corollary \ref{cor:surj}, the map from $\K_{[n+1]}[E']$ into $\K_{[n]}[E']$ induced by $D_f$ is surjective. Hence $\overline g$  is $0$ on $\K_{[n]}[E']$. Thus $g$ is zero on $[E']^n$ as claimed. 
\end{proof}

Going a step further, we get a weighted version of Gottlieb-Kantor theorem:
\begin{theorem} \label{kantorweight} Let $f:[E]^1\rightarrow \K$ and $g:[E]^n\rightarrow \K$. If $f$ does not take the value zero on a subset  of size at least  $2n+1$ and if $fg=0$ then  $g$ est identically zero on $[E]^n$.
\end{theorem}
\begin{proof}  Set $E':=supp(f)$. According to our hypothesis, $E'\not =\emptyset$.  We prove the lemma by induction on $n$. If $n=0$, pick $x\in E'$. We have $fg(\{x\})= f(\{x\})g(\emptyset)$. Since $f(\{x\})\not =0$ and $\K$ is an integral domain,  $g(\emptyset)=0$. Thus $g=0$ and the conclusion of the lemma holds for $n=0$. Let $n\geq 1$. Let $B\in [E]^n$.  We claim that $g(B)=0$.  Let  $F:=B\setminus E'$ and $r:=\vert F\vert $.
If  $r=0$, that is $B\subseteq E'$, we get $g(B)=0$  from Lemma \ref{lem:weak}.  If  $r \not = 0$, we define  $g_{r}$ on  $[E']^{n-r}$, setting  $g_{r}(B'):=g(B'  \cup F)$ for each $B'\in [E']^{n-r}$.   Let  $Q'\in [E']^{n-r+1}$ and $Q:= Q'\cup F$. We  have  $fg_{r}(Q')= \sum_{x\in Q'}  f(\{x\})g(Q\setminus \{x\})
=\sum_{x\in Q}  f(\{x\})g(Q\setminus \{x\})-\sum_{x \in F}  f(\{x\})g(Q\setminus \{x\})$.  From our hypothesis  on $g$ and the fact that $f(\{x\} )=0$ for all $x\notin E'$, both terms on the right hand side of the latter equality are $0$, thus $fg_{r}(Q')=0$. Since $\vert E'\vert \geq 2(n-r)+1$, induction on $n$ applies.  Hence $g_r$ is $0$ on $[E']^{n-r}$. This yields $g(B)=0$, proving our claim. Hence the conclusion of the lemma holds for $n$. 
\end{proof}
\begin{theorem} \label{tau(1,n)}$\tau(1,n)= 2n$ 
\end{theorem}
\begin{proof} 
Trivially, the formula holds if $n=0$. Hence, in the sequel, we suppose $n\geq 1$. 
\begin{claim}\label {claim: leq}$\tau(1,n)\leq 2n$ 
\end{claim}
\noindent {\bf Proof of Claim \ref{claim: leq}.}Let $f$ and  $g$ be non identically zero such that $fg=0$. From Theorem \ref{kantorweight}, the support $S$ of $f$ has at most 
$2n$ elements. From Lemma \ref{discharging} $\vert supp(f)\cup(supp(g)\vert \leq \vert S\vert +\tau(0,n)=\vert S\vert \leq 2n$.\endproof

For the converse inequality, we prove that:
\noindent\begin{claim}\label{claim:geq}There is a $2n$ element set $E$ and a map $g:[E]^n\rightarrow \K$  such that $eg=0$ and $g\not=0$.
\end{claim}
{\bf Proof of  Claim \ref{claim:geq}} Let $E:= \{0, 1\}\times \{0, \dots, n-1\}$. Set $E_{i}:= \{0,1\}\times \{i\}$ for $i<n$. Let $B\in [E]^n$. Set  $g(B):=0$ if $B$ is not a transversal of the $E_i$'s, $g(B):= -1$ if $B$ is a transversal  containing an odd number of elements of the form  $(0, i)$, $g(B)=1$ otherwise. Let $Q\in [E]^{n+1}$. If $Q$ is not a transversal of the $E_i$'s  then $g(B)=0$ for every $B\in [Q]^n$ hence $eg(Q)=0$. If $Q$ is a transversal,  then there is a unique index $i$ such that $E_i\subseteq Q$. In this case,  the only members of $[Q]^n $ on which $g$ is non-zero are $Q\setminus\{(0,i)\}$ and $Q\setminus \{(1,i)\}$; by our choice, they have opposite signs, hence $eg(Q)=0$.  
\endproof

Since in the example above $\tau (supp(e))= 2n$,  we have $\tau(1,n)\geq 2n$.  With this inequality, the proof of Theorem \ref{tau(1,n)} is complete. \end{proof}

\section{A lower bound for $\tau(m,n)$}



\begin{lemma}\label{lem:minor} $\tau(m,n)\geq (m+1)(n+1)-2$ for all $m,n\geq 1$.
\end{lemma}

\begin{proof}
For $m=1$ this inequality  was obtained in Claim \ref{claim:geq}. For the case $m>1$, we need the following improvement of Claim \ref{claim:geq}.

\begin{claim} \label{claim:homog}Let $n \geq 1$. There is a $2n$ element set $E$ and a map $g:[E]^n\rightarrow \K$  such that $eg=0$ and  $supp(g)=[E]^n$.
\end{claim}

{\bf Proof of Claim \ref{claim:homog}.}  Fix a $2n$-element set $E$. From 
Claim \ref{claim:geq}  and the fact that the symmetric group $\mathfrak S_E$ acts transitively on $[E]^n$, we get for each $B\in [E]^n$  some $g_B$ such that  $eg_B=0$ and $B\in supp(g_B)$.  Next, we observe that  a map $g:[E]^{n}\rightarrow \K$ satisfies $eg=0$ if and only if  $g$ belongs to the kernel of  the linear map $T: \K^{[E]^n}\rightarrow \K^{[E]^{n+1}}$ defined by setting $T(g)(Q):= \sum_{B\in  [Q]^{n}}g(B)$ for all $g$, $Q\in [E]^{n+1}$. To conclude, we apply the  claim below with $k:={2n \choose n}$.
\begin{claim}\label{claim:homog2} Let $e_1:= (1,0\dots,0),\dots, e_i:=(0, \dots, 1, 0, \dots, 0),\dots, e_k:= (0, \dots, 1)$ be the canonical basis of $\K^k$ and let $H$ be a subspace of $\K^k$. Then $H$ contains a vector with all its coordinates which are non-zero if and only if for every coordinate $i$ it contains some vector with this coordinate non-zero.\end{claim}
{\bf Proof of Claim \ref{claim:homog2}.} This assertion amounts to the fact that a vector space on an infinite field is not the union of finitely many proper subspaces. 
\endproof

Let $E$ be the disjoint union of $m$ sets $E_{0},\dots, E_{i}, \dots E_{m-1}$, each of  size $2n$. For each $i$, let $g_i:  [E_i]^{n}\rightarrow \K$ such that $\sum_{B\in [Q]^n}g_i(B)=0$ for each $Q\in [E_i]^{n+1}$ and $supp(g_i)= [E_i]^{n}$ (according to Claim \ref  {claim:homog}  such a $g_i$ exists).  Let $g:[E]^n \rightarrow \K$ be the "direct sum" of the $g_i$'s:
$g(B):=g_i(B)$ if $B\in [E_i]^n$, $g(B):=0$ otherwise. Let $f: [E]^m\rightarrow \K$ defined by setting $f(A):=1$ if $A\cap E_i\not =\emptyset$ for all $i<m$ and $0$ otherwise. 
Then, by a similar   argument as in Claim \ref{claim:geq},  $fg=0$. Next,  a  transversal of  $supp(f)$ must contains some $E_i$ (thus $\tau(supp(f))= 2n$).  And, also,  a transversal of $supp(g)$ must be a  transversal of each of the $supp(g_i)$'s. Since $supp(g_i)=[E_i]^n$, $\tau (supp(g_i))= n+1$, hence $\tau(supp(g))=(n+1)m$. We get easily that  $\tau(supp(f)\cup supp(g))= 2n +(n+1)(m-1)= mn+m+n-1$. This completes the proof  of Lemma \ref{lem:minor}.
\end{proof}
\begin{example}
The lemma above gives $\tau(2,2)\geq 7$.
 An example  illustrating this inequality is quite simple: let $E$ be made of  two squares, let $f$ be the map giving  value $-1/2$ on each side of the 
   squares, value $1$ on the diagonals; let $g$ be giving value $1$ on each pair meeting the two  squares.  Then for every $x\in E$, $E\setminus\{x\}$ is a minimal transversal of $supp(f)\cup supp(g)$. 
\end{example}

%




\end{document}